\documentclass[preprint]{imsart}

\pdfoutput=1
\RequirePackage[OT1]{fontenc}
\usepackage{amsthm, amsmath, natbib, amssymb, enumitem}
\RequirePackage[colorlinks,citecolor=blue,urlcolor=blue]{hyperref}

\usepackage{mathtools} 
\mathtoolsset{showonlyrefs}

\startlocaldefs
\numberwithin{equation}{section}
\theoremstyle{plain}
\newtheorem{theorem}{Theorem}[section]

\theoremstyle{remark}
\newtheorem{remark}{Remark}[section]

\newcommand{\rd}{\mathrm{d}}

\NewDocumentEnvironment{manual}{O{theorem}m}
 {%
  \addtocounter{theorem}{-1}%
  \renewcommand\thetheorem{#2}%
  \begin{#1}
 }
 {\end{#1}}

\endlocaldefs
\usepackage[margin=35truemm]{geometry}
\usepackage{booktabs}
\begin{document}
\begin{frontmatter}
\title{Minimax Simple Bayes Estimators of a Normal Variance}
\runtitle{Minimax Simple Bayes Estimators}

\begin{aug}
\author{\fnms{Yuzo} \snm{Maruyama}\thanksref{addr1,t1}\ead[label=e1]{maruyama@math.s.chiba-u.ac.jp}}

\runauthor{Y.~Maruyama}

\address[addr1]{Chiba University \printead{e1} 
}

\thankstext{t1}{supported by JSPS KAKENHI Grant Number 22K11933}
\end{aug}

\begin{abstract}
This paper is a follow-up to Maruyama and Strawderman (2006, Journal of Statistical Planning and Inference), which identified a new class of generalized Bayes estimators with a particularly simple form for estimating a normal variance under entropy loss. 
Although their previous work established the Bayesianity of these estimators, it did not provide a closed-form result for their minimaxity. 
In this paper, we revisit the problem and establish a definitive closed-form minimaxity result for this class of simple Bayes estimators.
\end{abstract}

\begin{keyword}[class=MSC]
\kwd[Primary ]{62C20}
\end{keyword}

\begin{keyword}
\kwd{Minimaxity}
\kwd{Bayes}
\kwd{Stein's paradox}
\kwd{Estimation of a variance}
\end{keyword}

\end{frontmatter}
\section{Introduction}
Let $ X $ and $ S $ be independent random variables,
 where $ X $ follows a
$p$-variate normal distribution $ \mathcal{N}_{p}(\theta,\sigma^{2}I_{p}) $ 
with unknown $ \theta $ and $ \sigma^2$ 
and $ S/\sigma^{2} $ follows a chi square distribution 
$ \chi_{n}^{2} $ with $ n $ degrees of freedom.
We consider the problem of estimating the variance
$ \sigma^{2} $ using an estimator $ \delta $ relative to 
the entropy loss function   
\begin{equation}\label{loss.variance}
 L(\delta,\sigma^{2})=\frac{\delta}{\sigma^{2}} - \log\frac{\delta}{\sigma^{2}}
-1.
\end{equation}
\cite{Stein-1964} showed that the best affine equivariant minimax estimator 
is $ \delta_{0}= S/n $ and
it can be improved by considering a class of scale equivariant
estimators 
\begin{equation} \label{form2}
 \delta_{\phi}= \{1-\phi(W)\}\frac{S}{n}, \quad\text{where}\quad W=\frac{\|X\|^2}{S}.
\end{equation}
Explicitly, he found an improved estimator 
\begin{equation}
 \delta^{S}=\min \Bigl(\frac{S}{n}, 
\frac{\|X \|^{2}+S}{p+n} \Bigr) =\{1-\phi^S(W)\}\frac{S}{n},\quad
\text{where}\quad\phi^{S}(w)= \max \Bigl(  0, \frac{p-n w}{p+n} \Bigr).
\end{equation}
This improvement can be regarded as Stein's paradox \citep{Stein-1956} for the estimation of a normal variance. 

The Stein paradox has traditionally been studied extensively 
in the context of estimating a mean vector.
For the $X$ and $S$ introduced at the beginning, 
consider the problem of estimating $\theta$ under the loss function
\begin{equation}\label{loss.mean}
L(\delta,\theta)= \frac{\|\delta-\theta\|^2}{\sigma^2}.
\end{equation}
The James–Stein estimator \citep{James-Stein-1961} is given by 
$(1-\alpha S/\|X\|^2)X $, which dominates the usual estimator $X$
when $p\geq 3$ and $0<\alpha \leq 2(p-2)/(n+2)$.
It is also known that certain classes of generalized Bayes estimators based on 
shrinkage-type priors dominate $X$.
In particular, \cite{Maru-Straw-2005} showed that 
\begin{equation}\label{simple.mean}
 \hat{\theta}=\left(1-\frac{\alpha}{\alpha+1+\|X\|^2/S}\right)X
\end{equation}
is a generalized Bayes estimator and that it dominates $X$ 
when $p\geq 3$ and $0<\alpha \leq 2(p-2)/(n+2)$.
Furthermore, for variance estimation under the same prior,
\cite{Maru-Straw-2006} showed that the simple estimator
\begin{equation}\label{simple.variance}
 \hat{\sigma}^2 =\left(1-\frac{\alpha}{\alpha+1+\|X\|^2/S}\right)\frac{S}{n}
\end{equation}
is also a generalized Bayes estimator.
Building on \cite{Straw-1974a},
\cite{Maru-Straw-2006} also derived sufficient conditions for this estimator to dominate $S/n$; 
however, those conditions were not obtained in closed form.
In this paper, we revisit this problem.

The organization of this paper is as follows. 
In Section \ref{sec:Bayes}, 
we review the Bayesianity of these estimators, \eqref{simple.mean} for $\theta$
and \eqref{simple.variance} for $\sigma^2$. 
In Section \ref{sec:minimax}, we demonstrate that
\begin{equation}
 \hat{\sigma}^2 =\left(1-\frac{\alpha}{\alpha+1+\|X\|^2/S}\right)\frac{S}{n}
\end{equation}
dominates $S/n$ when
\begin{equation}
 0<\alpha\leq \frac{-(n+2)+\sqrt{(n+2)^2+16p}}{2n}.
\end{equation}

\section{Bayesianity}
\label{sec:Bayes}
Let $\eta=1/\sigma^2$. For the general joint prior $\pi(\theta,\eta)$,
 the Bayes estimator of $\theta$ for the scaled quadratic loss \eqref{loss.mean} is
\begin{align}\label{hat.theta}
\hat{\theta}& =\frac{\iint \eta\theta\{\eta^{p/2}\exp(-\eta\|x-\theta\|^2/2)\}\{s^{n/2-1}\eta^{n/2}\exp(-\eta s/2)\}\pi(\theta,\eta)\rd \theta\rd\eta}
{\iint \eta
\{\eta^{p/2}\exp(-\eta\|x-\theta\|^2/2)\}\{s^{n/2-1}\eta^{n/2}\exp(-\eta s/2)\}
\pi(\theta,\eta)\rd \theta\rd\eta}\\
&=x+\frac{\iint \eta(\theta-x)\eta^{p/2+n/2}\exp(-\eta\|x-\theta\|^2/2-\eta s/2)\pi(\theta,\eta)\rd \theta\rd\eta}
{\iint \eta\eta^{p/2+n/2}\exp(-\eta\|x-\theta\|^2/2-\eta s/2)\pi(\theta,\eta)\rd \theta\rd\eta}\\
&=x+\frac{\nabla_x m(x,s)}{-2(\partial/\partial s)m(x,s)}
\end{align}
where
\begin{equation}\label{mxs}
 m(x,s)=\iint \eta^{p/2+n/2}\exp(-\eta\|x-\theta\|^2/2-\eta s/2)\pi(\theta,\eta)\rd \theta\rd\eta.
\end{equation}
Similarly the Bayes estimator of $\sigma^2=1/\eta$ for the entropy loss \eqref{loss.mean} is
\begin{align}\label{hat.sigma.2}
 \hat{\sigma}^2&= \frac{\iint \{\eta^{p/2}\exp(-\eta\|x-\theta\|^2/2)\}\{s^{n/2-1}\eta^{n/2}\exp(-\eta s/2)\}\pi(\theta,\eta)\rd \theta\rd\eta}
{\iint \eta\{\eta^{p/2}\exp(-\eta\|x-\theta\|^2/2)\}\{s^{n/2-1}\eta^{n/2}\exp(-\eta s/2)\}
\pi(\theta,\eta)\rd \theta\rd\eta}\\
&=\frac{m(x,s)}{-2(\partial/\partial s)m(x,s)}.
\end{align}
Following \cite{Maru-Straw-2005,Maru-Straw-2006}, we adopt a specific prior.
Let the conditional distribution of $ \theta $ given $ \lambda $ and $ \eta$, 
for $ 0 < \lambda <1 $, be normal with mean vector $ 0 $ and covariance matrix 
$ \lambda^{-1}(1-\lambda) \eta^{-1} I_{p} $ and  
let the density functions of $ \lambda $ and $ \eta $ be proportional to 
$ \lambda^{a}(1-\lambda)^{n/2-1}I_{(0,1)}(\lambda) $ and 
$ \eta^a I_{(0,\infty)}(\eta) $, respectively.
The resulting improper prior density is
\begin{equation}\label{my.prior}
\pi(\theta,\eta)=\frac{1}{(2\pi)^{p/2}} \left( \frac{\eta \lambda}{1-\lambda} \right)^{p/2}  
\exp \left(-\frac{\lambda}{1-\lambda}\frac{\eta}{2}\| \theta \| ^{2} \right)  
 \eta^{a} \lambda^{a}(1-\lambda)^{n/2}I_{(0,1)}(\lambda)I_{(0,\infty)}(\eta) .
\end{equation}
By substituting this into \eqref{mxs} and using the identity
\begin{equation}
 \|x-\theta\|^2+\frac{\lambda}{1-\lambda}\|\theta\|^2
=\frac{1}{1-\lambda}\|\theta-(1-\lambda)x\|^2+\lambda \|x\|^2,
\end{equation}
we obtain 
\begin{align*}
&m(x,s)\\
&= \iint \eta^{p/2+n/2+a} 
\exp \left(-\eta\frac{\lambda\| x \| ^{2}+s}{2}\right) 
 \lambda^{p/2+a}(1-\lambda)^{n/2-1}\rd\eta \rd\lambda \\
& =  \Gamma(p/2+n/2+a+1)2^{p/2+n/2+a+1}
\int_0^1 \lambda^{p/2+a}(1-\lambda)^{n/2-1}(\lambda\|x\|^2+s)^{-p/2-n/2-a-1}
\rd\lambda.
\end{align*}
Furthermore, by a following identity,
which is given by the change of variables
$ t=(1+w)\lambda/(1+w\lambda)$
\begin{equation}
 \int_{0}^{1}\lambda^{\alpha}(1-\lambda)^{\beta}(1+w \lambda)^{-\gamma}\rd  \lambda
= \frac{1}{(w+1)^{\alpha+1}}\int_{0}^{1}t^{\alpha}(1-t)^{\beta}
\left\{1-\frac{tw}{w+1} \right\}^{-\alpha-\beta+\gamma-2}\rd t,
\end{equation}
we have
\begin{align}
 m(x,s) &=  \frac{\Gamma(p/2+n/2+a+1)2^{p/2+n/2+a+1}}{s^{p/2+n/2+a+1}}\frac{B(p/2+a+1,n/2)}{(1+\|x\|^2/s)^{p/2+a+1}}\\
 &=  \frac{c}{s^{n/2}(\|x\|^2+s)^{p/2+a+1}}, 
\end{align}
where $c=\Gamma(p/2+n/2+a+1)2^{p/2+n/2+a+1}B(p/2+a+1,n/2)$.
Note
\begin{align}\label{mxs.2}
 \nabla_x  m(x,s)&=2c(-p/2-a-1)s^{-n/2}(\|x\|^2+s)^{-p/2-a-2}x\\ 
&=\frac{2(-p/2-a-1)}{(\|x\|^2+s)}m(x,s)
\end{align}
and
\begin{align}\label{mxs.3}
 &\frac{\partial}{\partial s}m(x,s)\\
&=c(-p/2-a-1)s^{-n/2}(\|x\|^2+s)^{-p/2-a-2} -(n/2)s^{-n/2-1}(\|x\|^2+s)^{-p/2-a-1}\\
&=-\left(\frac{p/2+a+1}{\|x\|^2+s}+\frac{n/2}{s}\right)m(x,s).
\end{align}
Then, by \eqref{hat.theta}, \eqref{hat.sigma.2}, \eqref{mxs.2} and \eqref{mxs.3}, 
the Bayes estimators reduce to
\begin{align}\label{theta.sigma.2}
 \hat{\theta} &=\left(1-\frac{\alpha}{\alpha+1+\|X\|^2/S}\right)X,\qquad
 \hat{\sigma}^2 =\left(1-\frac{\alpha}{\alpha+1+\|X\|^2/S}\right)\frac{S}{n}\\
&\qquad \qquad \qquad\qquad \qquad \qquad\text{where}
\quad \alpha=\frac{p/2+a+1}{n/2} .
\end{align}

\section{Minimaxity}
\label{sec:minimax}
We now establish a closed-form minimaxity result for the estimator $\hat{\sigma}^2$
given by \eqref{theta.sigma.2}.
\begin{theorem} \label{newstrawl1}
The simple Bayes estimator
\begin{equation}\label{simple.variance.1}
 \hat{\sigma}^2 =\left(1-\frac{\alpha}{\alpha+1+\|X\|^2/S}\right)\frac{S}{n}
\end{equation}
dominates $S/n$ when
\begin{equation}\label{theorem.alpha}
 0<\alpha\leq \alpha\coloneq\frac{-(n+2)+\sqrt{(n+2)^2+16p}}{2n}.
\end{equation}
\end{theorem}

\begin{proof}
Let $\alpha_*$ be as in \eqref{theorem.alpha}, which can also be expressed as
\begin{equation}
\alpha_*
=\max_{\kappa\in(0,1)}\min
\left(\frac{1}{\kappa}-1,\frac{4p\kappa}{n(n+2\kappa)}\right).
\end{equation}
Write $W=\|X\|^2/S$ and $ V=S/\sigma^{2}$. 
The difference in risks between 
$ S/n$ and $ \hat{\sigma}^2$ equals
\begin{align} 
 \Delta&=R(\theta,\sigma^{2},S/n)- R(\theta,\sigma^{2},\hat{\sigma}^2)\\
&=E \left[ \left\{\frac{S}{n\sigma^2}-\log\left(\frac{S}{n\sigma^2}\right)-1\right\}\right. \\
&\qquad \left. -\left\{ \left( 1-\frac{\alpha}{\alpha+1+W}\right)\frac{S}{n\sigma^2}-
\log\left(\left( 1-\frac{\alpha}{\alpha+1+W}\right)\frac{S}{n\sigma^2}\right)-1\right\}\right] \\
&=\frac{\alpha}{n}E \left[ \frac{V}{\alpha+1+W}\right]+
E \left[\log \left( 1-\frac{\alpha}{\alpha+1+W} \right)\right],
\end{align} 
In the following we will show $\Delta\geq 0$ for $\alpha\in(0,\alpha]$.

Let $g(w)=1/(\alpha+1+W)$.
Using a Poisson mixture representation with $J\sim \mathrm{Po}(\| \theta \|^{2}/(2\sigma^{2}) )$
and $U_j\sim\chi^2_{p+2j}$, we have  
\begin{align}
 E \left[ Vg(W)\right]
&= E^J\left[ E\left[ Vg(U_j/V)\mid J=j\right]\right]
= E^J\left[ E\left[ (U_i+V)\frac{g(U_j/V)}{1+U_j/V}\mid J=j\right]\right]\\
&=E^J\left[ E\left[U_i+V\mid J=j\right]
E\left[\frac{g(U_j/V)}{1+U_j/V}\mid J=j\right]\right]\\
&=E^J\left[ (p+n+2j)E\left[\frac{g(U_j/V)}{1+U_j/V}\mid J=j\right]\right],
\end{align}
where the third equality follows from the independence of $U_i+V$ and $U_i/V$.
Further, we have
\begin{equation}\label{log.lower} 
 \log(1-x)=-\sum_{i=1}^{\infty}\frac{x^{i}}{i} \geq -x 
-\frac{1}{2}\frac{x}{1-x}x,
\end{equation}
and hence
\begin{equation}\label{log.lower.1} 
\log \left( 1-\frac{\alpha}{\alpha+1+w} \right)\geq 
-\frac{\alpha}{\alpha+1+w}-\frac{1}{2}\frac{\alpha}{1+w}\frac{\alpha}{\alpha+1+w}.
\end{equation}
For $W_j=U_j/V$, we have
\begin{align}\label{Delta.1}
 \Delta&\geq
\frac{\alpha}{n}E^J\left[ 
E\left[\frac{p+n+2j}{(\alpha+1+W_j)(1+W_j)}-
\frac{n}{\alpha+1+W_j}-\frac{1}{1+W_j}\frac{n/2}{\alpha+1+W_j}\mid J=j\right]\right]\\
&=\frac{\alpha}{n}E^J\left[ 
E\left[\frac{(1+W_j)^\epsilon}{\alpha+1+W_j}k_j(W_j)\mid J=j\right]\right],
\end{align}
where 
\begin{equation}\label{eps.alp}
 \epsilon=\frac{1}{1+\alpha}
\end{equation}
and 
\begin{align}
 k_j(w_j)&=\frac{p+n+2j}{(1+w_j)^{\epsilon+1}}-
\frac{n}{(1+w_j)^\epsilon}-\frac{\alpha n/2}{(1+w_j)^{1+\epsilon}}\\
&=\frac{1}{(1+w_j)^\epsilon}\left(
\frac{p+n+2j}{1+w_j}-n-\frac{\alpha n/2}{1+w_j}\right).
\end{align}
The function $ (1+w)^\epsilon/(\alpha+1+w)$ in \eqref{Delta.1} is monotone non-increasing since
\begin{align}\label{monotone}
\frac{\rd}{\rd w}\log \frac{(1+w)^\epsilon}{\alpha+1+w}
&=\frac{\epsilon}{1+w}-\frac{1}{\alpha+1+w}
=\frac{\epsilon(\alpha+1)-1+w(\epsilon-1)}{(1+w)(\alpha+1+w)}\\
&=-\frac{w(1-\epsilon)}{(1+w)(\alpha+1+w)}\\ &\leq 0.
\end{align}
The value $k_j(w_j)$ at $w_j=0$ is
\begin{align}
 k_j(0)&=p+2j-\alpha\frac{n}{2}\\ &\geq p-\alpha_*\frac{n}{2}\\
&= p -\frac{-(n+2)+\sqrt{(n+2)^2+16p}}{4}\\
&=\frac{4p+n+2-\sqrt{(n+2)^2+16p}}{4},
\end{align}
which is positive. Hence 
$k_j(w_j)$ changes sign only once from positive to negative on $ (0,\infty) $.
Let $w_j^*$ satisfy $k_j(w_j^*)=0$.
Then we have
\begin{equation}\label{monotone.1}
 E\left[
\frac{(1+W_j)^\epsilon}{\alpha+1+W_j}k_j(W_j)\mid J=j\right]
\geq \frac{(1+w_j^*)^\epsilon}{\alpha+1+w_j^*}
 E\left[k_j(W_j)\mid J=j\right].
\end{equation}
In the above, we have
\begin{align}
 & E\left[k_j(W_j)\mid J=j\right]\label{Delta.2}\\
&=
(p+n+2j)B(p/2+j,n/2+\epsilon+1)-nB(p/2+j,n/2+\epsilon)-
(\alpha n/2)B(p/2+j,n/2+\epsilon+1)\notag\\
&=B(p/2+j,n/2+\epsilon)
\left\{(p+n+2j)\frac{n/2+\epsilon}{p/2+j+n/2+\epsilon}
-n-\frac{\alpha n}{2}\frac{n/2+\epsilon}{p/2+j+n/2+\epsilon}
\right\}\notag\\
&\geq B(p/2+j,n/2+\epsilon)
\left\{(p+n)\frac{n+2\epsilon}{p+n+2\epsilon}
-n-\frac{\alpha n}{2}\frac{n+2\epsilon}{p+n+2\epsilon}
\right\}\notag\\
&= \frac{B(p/2+j,n/2+\epsilon)}{p+n+2\epsilon}
\left\{(p+n)(n+2\epsilon)-n(p+n+2\epsilon)- \frac{\alpha n}{2}(n+2\epsilon)\right\}\notag\\
&= \frac{n(n+2\epsilon)B(p/2+j,n/2+\epsilon)}{2(p+n+2\epsilon)}
\left\{\frac{4p\epsilon}{n(n+2\epsilon)}- \alpha \right\}.\notag
\end{align}
Finally, by \eqref{eps.alp}, we have
\begin{align}
 \frac{4p\epsilon}{n(n+2\epsilon)}- \alpha
&=\frac{4p/n}{n(\alpha+1)+2}-\alpha 
=\frac{4p/n-\alpha n(\alpha+1)+2}{n(\alpha+1)+2} \label{Delta.3}\\
&\geq \frac{4p/n-\alpha_* n(\alpha_*+1)+2}{n(\alpha+1)+2}\\ &=0,
\end{align}
where the inequality follows from $\alpha\leq \alpha_*$.
Therefore, by \eqref{Delta.1}, \eqref{monotone.1}, \eqref{Delta.2}, and \eqref{Delta.3}, we 
establish $\Delta\geq 0$ for $0<\alpha\leq \alpha_*$.
\end{proof}
\begin{remark}
The core methodology of the proof presented above relies on the inequality in \eqref{log.lower} 
and the monotonicity of $(1+w)^\epsilon / (\alpha+1+w)$, 
as applied in \eqref{monotone} and \eqref{monotone.1}. 
These techniques originate from the work of \cite{Maru-Straw-2006} 
and \cite{Straw-1974a}, respectively. 
However, those earlier studies focused on establishing broad sufficient conditions 
for minimaxity that could be applied across a wide class of estimators. 
Consequently, the specific structural advantages of individual estimators were not fully exploited.
In contrast, for the estimator \eqref{simple.variance.1} analyzed in this paper, 
we leverage its unique form-specifically, 
the fact that the term $x/(1-x)$ in  \eqref{log.lower} 
simplifies to $\alpha/(w+1)$ in \eqref{log.lower.1}. 
This specific exploitation of the estimator's structure provides a substantial advantage, 
allowing for the derivation of the closed-form minimaxity result presented here.
\end{remark}

\end{document}